\documentclass[12pt,reqno]{amsart}
\usepackage{dsfont}
\usepackage{tikz}
\usepackage{amssymb}
\usepackage{amsmath}
\usepackage{enumerate}
\usepackage{mathabx,epsfig}
\usepackage{tikz-cd}

\usepackage{changepage}

\newtheorem{fact}{Fact}[section]
\newtheorem{lemma}[fact]{Lemma}
\newtheorem{theorem}[fact]{Theorem}

\newtheorem{rremark}[fact]{Remark}
\newenvironment{remark}{\begin{rremark} \rm}{\end{rremark}}

\newtheorem{corollary}[fact]{Corollary}

\DeclareMathOperator\Id{Id}

\DeclareMathOperator\Supp{Supp}
\DeclareMathOperator\Hom{Hom}
\DeclareMathOperator\End{End}
\DeclareMathOperator\Span{Span}
\author{Samantha Moore}
\address{Department of Mathematics, University of North Carolina at Chapel Hill, USA}
\email{scasya@live.unc.edu}

\usepackage{lipsum}

\newcommand\blfootnote[1]{%
  \begingroup
  \renewcommand\thefootnote{}\footnote{#1}%
  \addtocounter{footnote}{-1}%
  \endgroup
}

\title{Hyperplane Restrictions of Indecomposable $n$-Dimensional Persistence Modules}

\begin{document}
\blfootnote{This material is based upon work supported
by the National Science Foundation Graduate Research Fellowship under Grant No. 1650116.}

\begin{abstract}
Understanding the structure of indecomposable $n$-dimensional persistence modules is a difficult problem, yet is foundational for studying multipersistence. To this end, Buchet and Escolar showed that any finitely presented rectangular $(n-1)$-dimensional persistence module with finite support is a hyperplane restriction of an $n$-dimensional persistence module. We extend this result to the following: If $M$ is any finitely presented $(n-1)$-dimensional persistence module with finite support, then there exists an indecomposable $n$-dimensional persistence module $M'$ such that $M$ is the restriction of $M'$ to a hyperplane. We also show that any finite zigzag persistence module is the restriction of some indecomposable $3$-dimensional persistence module to a path.
\end{abstract}
\maketitle
\section{Introduction}

Understanding the structure of multiparameter persistence modules is an important foundational problem in the area of persistent homology. Every finite multiparameter persistence module has a decomposition into indecomposable persistence modules. Thus, to understand the structure of all multiparameter persistence modules, it suffices to understand the structure of the indecomposables. But the structure of the indecomposable $n$-dimensional persistence modules is rich when $n>1$; such modules have no complete discrete invariant [CZ]. One might still hope that finite indecomposeable $n$D modules would have simple 'substructures' in some sense. For example, it would be convenient if the hyperplane restrictions of such modules are in some restricted set of $(n-1)$-dimensional persistence modules. Buchet and Escolar proved that this is not true when $n=2$; any finite $1$-dimensional persistence module is a hyperplane restriction of an indecomposeable $2$D persistence module (Theorem 1.1 of [BE1]). Their work was motivated largely by [LW], which introduced the software RIVET which is an important tool for studying $2$D persistence modules. In particular, one of RIVET's functions is the ability to find the barcode of any path in a finite $2$D persistence module. 

Buchet and Escolar further found that any finite rectangular $n$D persistence module $M$ is a hyperplane restriction of some indecomposeable $(n+1)$ dimensional persistence module $M'$ (Theorem 7.7 of [BE1]). To prove this result, Buchet and Escolar provided three constructions for $M'$. In this paper, we extend their result to further shed light on how rich the structure of $(n+1)$-dimensional persistence modules can be. In particular, our primary results are as follows:

\begin{itemize}
\item{\textbf{Theorem 3.1} Let $M$ be a finite $n$-dimensional interval persistence module. Then there exists an indecomposable $(n+1)$-dimensional persistence module $M'$ such that $M$ is the restriction of $M'$ to a hyperplane. }

 \item{\textbf{Corollary 3.2} Let $M=\bigoplus\limits_{i=1}^m K[\alpha_i, \beta_i]$ be a finite zigzag persistence module. Then there exists an indecomposable $3$-dimensional persistence module $M'$ such that $M$ is the restriction of $M'$ to a path.}

\item{\textbf{Theorem 3.3} Let $M$ be any finite $n$-dimensional persistence module. Then there exists an indecomposable $(n+1)$-dimensional persistence module $M'$ such that $M$ is the restriction of $M'$ to a hyperplane. }
\end{itemize}

Though Theorem 3.3 is a direct generalization of Theorem 3.1, we state and prove these results separately because we have three constructions for $M'$ in the setting of Theorem 3.1 (based on Buchet and Escolar's three constructions for finite rectangular $n$D persistence modules), but only one of these constructions can be adapted to the setting of Theorem 3.3.

The structure of the paper is as follows: Sections 2.1 and 2.2 cover the basic definitions needed from multipersistence, as well as an important lemma about the homomorphisms between interval multiparameter persistence modules. In Section 2.3, we discuss the classification of $n$D persistence modules and state the constructions/proof for Theorem 7.7 of [BE1], as our results rely heavily on these. In Section 3, we provide constructions and proofs for our results stated above. In Section 4, we discuss potential future directions for research.

\section{Background and Previous Results}

\subsection{Initial Definitions}

Let $\{e_j\}_{j=1}^n$ denote the standard basis of $\mathbb{N}^n$. An \textbf{$n$-dimensional persistence module $M$} over a field $ K$ is an assignment of a vector space $M_\alpha$ to each $\alpha\in\mathbb{N}^n$ and a homomorphism $\phi^M_{\alpha, \alpha+e_j}:M_\alpha\rightarrow M_{\alpha+e_j}$ for each $\alpha\in\mathbb{N}^n$ and each $j\in[1,n]$ such that the resulting diagram commutes. If $\alpha=(\alpha_1, \alpha_2,..., \alpha_n)$ and $\beta=(\beta_1, \beta_2,...,\beta_n)$ are points in $\mathbb{N}^n$ such that $\alpha_j\leq \beta_j$ for all $j$, then we write $\alpha\leq\beta$. This yields a partial ordering on $\mathbb{N}^n$. Notice that  definition of an $n$-dimensional persistence module $M$ implies that there is a well-defined map $\phi_{\alpha,\beta}^M:M_\alpha\rightarrow M_\beta$ whenever $\alpha\leq\beta$.

A \textbf{homomorphism} $f:M\rightarrow N$ between $n$-dimensional persistence modules $M$ and $N$ is a collection of homomorphisms $f_\alpha : M_\alpha\rightarrow N_\alpha$ such that $\phi^N_{\alpha, \alpha+e_j}\circ f_\alpha = f_{\alpha+e_j} \circ \phi^M_{\alpha, \alpha+e_j}$ for all $\alpha\in\mathbb{N}^n$ and all $j\in [1,n]$.

Let $V=\{v_i\}_i$ be a set of homogeneous elements in $M$ (meaning that for each $i$, $v_i\in M_{\alpha_i}$ for some $\alpha_i\in\mathbb{N}^n$). We say that $V$ is a \textbf{generating set} of $M$ if for all $\beta\in\mathbb{N}^n$ and $v\in M_\mathbb{\beta}$, there exist coefficients $c_i\in K$ such that $v=\bigoplus\limits_i c_i\phi_{\alpha_i,\beta}^M(v_i)$. If $M$ has a generating set $V$ with finite cardinality, then we say that $M$ is \textbf{finitely generated}. 
All of the $n$D persistence modules in this paper are assumed to be \textbf{finite}, meaning they are finitely presented and have finite support (where the \textbf{support} of $M$, denoted by $\Supp(M)$, is defined to be $\{\alpha\in\mathbb{N}^n| M_\alpha\neq 0\}$). %[[[[does this match the typical def of finite??]]]

There is a bijective correspondence between finitely generated $n$-dimensional persistence modules $M$ and finitely generated modules over the polynomial ring $ K[x_1, x_2,..., x_n]$; given a  finitely generated $n$-dimensional persistence module $M$, define the finitely generated module $M':=\bigoplus\limits_\alpha M_\alpha$ over $ K[x_1, x_2,..., x_n]$ with $ K[x_1, x_2, \cdots, x_n]$-action given by $x_j \cdot v:= \phi_{\alpha, \alpha+e_j} (v)$ for all $v\in M_\alpha$ and $j\in[1,n]$. This bijection is in fact an equivalence of categories [CZ], which allows us to utilize definitions and results from commutative algebra.

The \textbf{direct sum} of two $n$-dimensional persistence modules $M$ and $N$ is the $n$-dimensional persistence module $M\bigoplus N$, whose vector spaces are given by $(M\bigoplus N)_\alpha:=M_\alpha\bigoplus N_\alpha$ and whose maps are given by $\phi^{M\bigoplus N}_{\alpha,\beta}=\phi^{M}_{\alpha,\beta}\bigoplus \phi^{N}_{\alpha,\beta}$ for all $\alpha\leq\beta\in\mathbb{N}^n$. An $n$-dimensional persistence module $M$ is \textbf{indecomposable} if $M$ cannot be written as $M=M_1\bigoplus M_2$ where $M_1,M_2$ are both nonzero $n$-dimensional persistence modules. By the Krull-Schmidt Theorem, any finite $n$D persistence module $M$ has an \textbf{indecomposable decomposition}, meaning there exist indecomposable $n$-dimensional persistence modules $I_i$ such that $M=\bigoplus\limits_{i=1}^m I_i$. The following lemma will be important for proving our main results:
\begin{lemma}[Corollary 4.8 of {[A]}]
Let $M$ be a finite $n$-dimensional persistence module. Then $\End(M)$ is local if and only if $M$ is indecomposable.
\end{lemma}

\subsection{Interval $n$-dimensional Persistence Modules}

Let $M=\bigoplus\limits_{i=1}^m I_i$ be the decomposition of an $n$-dimensional persistence module $M$ into its indecomposable summands. Then $M$ is called an \textbf{interval} $n$-dimensional persistence module if $\dim((I_i)_\alpha)\in\{0,1\}$ for all $\alpha\in\mathbb{N}^n$ and 
\[   \phi^{I_i}_{\alpha,\beta}=\left\{
\begin{array}{ll}
      \Id_ K & \text{if }  \alpha\leq\beta \text{ and } (I_i)_\alpha=(I_i)_\beta= K \\
      0 & \text{otherwise}. \\
 
\end{array}
\right. \] Interval $n$D persistence modules are some of the simplest multiparameter persistence modules, but they will be fundamental to proving the results in Sections 2.3 and 3. Thus, we use this section to explore the properties of such modules. For additional properties of interval multiparameter persistence modules, see [ABENY, AENY, DX].

\begin{remark}
Let $I$ be an indecomposable $n$-dimensional interval persistence module. We claim that $\alpha,\beta \in \Supp(I)$ implies $\gamma\in\Supp(I)$ whenever $\alpha\leq\gamma\leq\beta$. Indeed, $\alpha,\beta \in \Supp(I)$ implies that $\phi^{I}_{\alpha,\beta}=\Id_ K$ by above. Let $\gamma\in\mathbb{N}^n$ such that $\alpha\leq\gamma\leq\beta$. Then $\Id_ K=\phi^{I}_{\alpha,\beta}=\phi^{I}_{\gamma,\beta}\circ\phi^{I}_{\alpha,\gamma}$ since the $\phi^{I}$ maps must commute with each other. Thus $\phi^{I}_{\alpha,\gamma}\neq 0$, implying (by the definition of interval $n$D persistence modules) that $\gamma\in \Supp(I)$.
\end{remark}

An indecomposable $n$D interval persistence module $I$ is \textbf{rectangular} if its support is an $m$-dimensional box for some $m\leq n$. Such modules are denoted by $ K[\alpha, \beta]$ where $\alpha, \beta\in \Supp (I)$ such that $\alpha\leq\gamma\leq\beta$ for all $\gamma\in\Supp (I)$. An $n$-dimensional persistence module $M$ is \textbf{rectangular} if each of its indecomposeable summands is.

Suppose $M$ and $N$ are finite $n$-dimensional interval persistence modules. Let $C_i$ be an edge-connected component of $\Supp(M)\cap\Supp(N)$. Call $C_i$ \textbf{non-viable} if there exists $\beta\in\Supp(M)\setminus \Supp(N)$ and $\alpha\in C_i$ such that $\beta<\alpha$, or if  there exists $\beta\in\Supp(N)\setminus\Supp(M)$ and $\alpha\in C_i$ such that $\alpha<\beta$. Otherwise, call $C_i$ \textbf{viable}. For an example of these concepts, see Fig. \ref{fig:Viable}.

\begin{figure}[h]
    \centering
    \includegraphics[scale=.24]{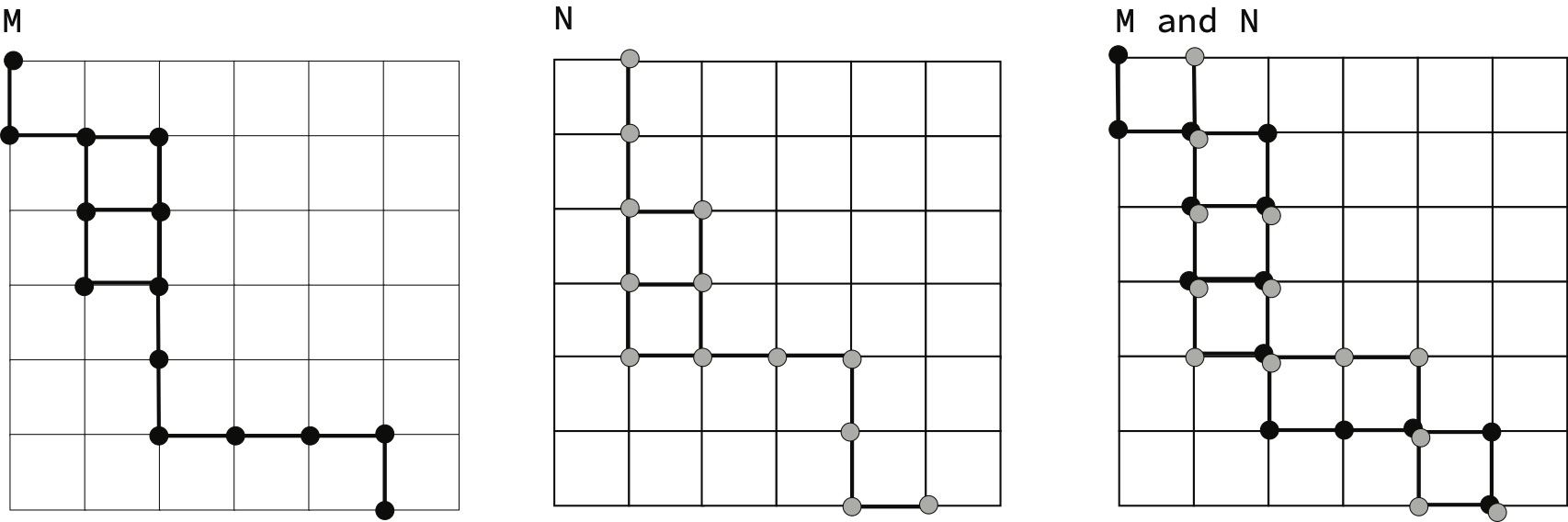}
    \vspace{.2cm}
    \caption{The first two images in this figure are $2$-dimensional interval persistence modules $M$ and $N$. In each module, each dot at $\alpha\in\mathbb{N}^2$ represents a basis element of the $\alpha$-vector space. Thick edges represent the identity map on colored components (i.e. gray basis elements map to gray basis elements while black basis elements map to black basis elements.), while the thinner edges represent the zero map. The third image shows $M$ and $N$ overlaid, which makes it clear that there are three components in $\Supp(M)\cap\Supp(N)$. Of these components, only the rightmost is viable. The leftmost component $C$ is non-viable because $(1,6)\in\Supp(N)\setminus\Supp(M)$ and $(1,5)\in C$ have the property that  $(1,5)<(1,6)$. Meanwhile, the middle component $C$ is non-viable because $(2,1)\in\Supp(M)\setminus \Supp(N)$ and $(2,2)\in C$ have the property that $(2,1)<(2,2)$.}
    \label{fig:Viable}
\end{figure}

\begin{lemma}
Suppose $M$ and $N$ are finite $n$-dimensional indecomposeable interval persistence modules. Let $m$ denote the number of viable components of $Supp(M)\cap\Supp(N)$.  Then \[\Hom(M,N)\cong\left\{
\begin{array}{ll}
     K^m & \text{if } m>0 \\
      0 & \text{otherwise}. \\
 
\end{array}
\right. \]
\end{lemma}

\begin{proof}
Let $\{C_i\}_i$ denote the set of (possibly non-viable) components of $\Supp(M)\cap\Supp(N)$ and suppose $f\in\Hom(M,N)$. Then $f|_\alpha$ is trivially zero whenever $\alpha\notin\Supp(M)\cap\Supp(N)$. On the other hand, $f|_\alpha=c_\alpha \Id_ K$ for some $c_\alpha\in K$ whenever $\alpha\in\Supp(M)\cap\Supp(N)$, as $M_\alpha= N_\alpha= K$. Furthermore, because $f$ must commute with the maps $\phi^M$ and $\phi^N$, it must be the case that $c_\alpha=c_\beta$ whenever $\alpha,\beta$ are in the same component of $\Supp(M)\cap\Supp(N)$. That is, $f$ is fully determined by $\{c_i\in K\}_{i}$, where $c_i:=c_\alpha$ if $\alpha\in C_i$.

We claim that if $C_i$ is non-viable, then $c_i$ must be zero. That is, for all $\alpha\in C_i$, $f|_\alpha$ is the zero map. Suppose there exists $\beta\in\Supp(M)\setminus\Supp(N)$ and $\alpha\in C_i$ such that $\beta<\alpha$. Then $\beta\notin \Supp(N)$ implies $f|_{\beta}$ is the zero map. Because $f$ must commute with the maps $\phi^M, \phi^N$, it must be the case that $f|_\alpha\circ \phi^M_{\beta,\alpha}=\phi^N_{\beta,\alpha}\circ f|_{\beta}=0$. We know that $\phi^M_{\beta,\alpha}=\Id_{ K}$ since $M$ is an indecomposable interval $n$-dimensional persistence module. Thus it must be the case that $f|_\alpha=0$. Similarly, if there exists $\beta\in\Supp(N)\setminus\Supp(M)$ and $\alpha\in C_i$ such that $\alpha<\beta$, then $M_\beta=0$ implies $f|_\beta=0$. Commutativity of the $f, \phi^M$, and $\phi^N$ maps yields that $f|_\alpha=f|_\beta=0$.

Alternatively, if $C_i$ is viable, then we claim that the $f, \phi^M$, and $\phi^N$ maps are able to commute even if $c_i$ is nonzero. Let $\alpha\in C_i$. Notice that $\alpha+e_j$ must be an element of $C_i$ or $\mathbb{N}^n\setminus\Supp(N)$ since $C_i$ is viable. If $\alpha+e_j \in C_i$ then $f|_{\alpha+ej}\circ \phi^M_{\alpha,\alpha+e_j}=c_i \Id_ K= \phi^N_{\alpha,\alpha+e_j}\circ f|_{\alpha}$, as desired. On the other hand, if $\alpha+e_j \in \mathbb{N}^n\setminus\Supp(N)$, then $f|_{\alpha+ej}\circ \phi^M_{\alpha,\alpha+e_j}$ and $\phi^N_{\alpha,\alpha+e_j}\circ f|_{\alpha}$ are maps with codomain equal to zero, so both compositions are trivially the zero map.

Thus $c_i$ may be nonzero whenever $C_i$ is a viable component, implying that any $f\in\Hom(M,N)$ is determined by $\{c_i\in K| C_i \text{ is viable and } f|_\alpha=c_i \Id_ K \text{ fo all }\alpha\in C_i\}$.
\end{proof}

We will primarily work with pairs of indecomposable interval $n$-dimensional persistence modules $I$ and $J$ such that $\Hom(I,J)\cong  K$ or $0$. In the setting that $\Hom(I,J)\cong K$, we use the notation $f_{I}^{J}$ to indicate a fixed natural nonzero homomorphism from $I$ to $J$.

\subsection{The Structure of Indecomposeables}

Recall that every finite $n$D persistence module has an indecomposable decomposition. Thus, to understand the structure all $n$-dimensional persistence modules, it suffices to understand the set of indecomposable $n$-dimensional persistence modules. When $n=1$, the indecomposeable persistence modules are of a simple form; $M$ is indecomposable if and only if there is an interval $[\alpha,\beta]$ (with $\beta$ possibly infinite) such that
\[   \dim(M_\gamma)=\left\{
\begin{array}{ll}
    1 & \text{if }  \alpha\leq\gamma\leq\beta  \\
      0 & \text{otherwise}\\
 
\end{array}
\right. \]

and \[   \phi^{M}_{\gamma,\delta}=\left\{
\begin{array}{ll}
      \Id_ K & \text{if }  \alpha\leq\gamma\leq\delta\leq\beta \\
      0 & \text{otherwise} \\
 
\end{array}
\right. \] [G]. That is, every indecomposable $1$-dimensional persistence module is rectangular. Unfortunately, for $n>1$, the full set of indecomposable $n$-dimensional persistence modules do not have such a simple structure.

\begin{theorem}[CZ] There does not exist a complete discrete invariant for the set of indecomposable $n$-dimensional persistence modules whenever $n>1$.
\end{theorem}

In other words, there does not exist a finitely parameterized invariant that can distinguish between all indecomposeable $n$-dimensional persistence modules. This was proved in [CZ] by showing the existence of a continuously parameterized family of non-isomorphic indecomposable $n$-dimensional persistence modules. Buchet and Escolar have found other continuously parameterized families of non-isomorphic indecomposable $n$D persistence modules, such as the family shown in Fig. \ref{fig:IndecomposableExample} [BE2]. The following theorem gives insight to how complicated the structure of indecomposable $(n+1)$-dimensional persistence modules can be.

\begin{figure}[h]
    \centering
    \includegraphics[scale=.4]{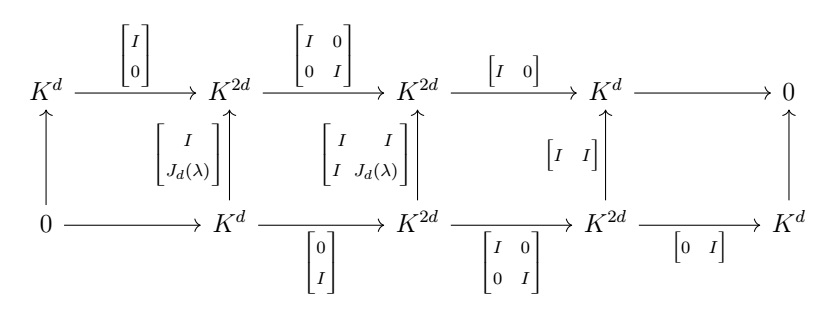}
    \vspace{.2cm}
    \caption{For each choice of $d\in\mathbb{N}$, shown is a continuous family of non-isomorphic indecomposeable $2$-dimensional persistence modules, where continuity comes from allowing $\lambda\in K=\mathbb{R}$ to vary. The maps $\phi^M$ are written in matrix form with respect to the bases $ K[1,4]^d\bigoplus K[2,3]^d$ of layer $0$ and $ K[0,3]^d\bigoplus K[1,2]^d$ of layer $1$. The matrix denoted by $J_{d}(\lambda)$ is the $d\times d$ Jordan block with $\lambda$ along its diagonal. This family was introduced in [BE2].}
    \label{fig:IndecomposableExample}
\end{figure}

\begin{theorem}[BE1]
Let $M=\bigoplus\limits_{i=1}^m K[\alpha_i, \beta_i]$ be a finite $n$-dimensional rectangular persistence module. Then there exists an indecomposeable finite $(n+1)$-dimensional persistence module $M'$ such that $M$ is a hyperplane restriction of $M'$.
\end{theorem}

In fact, Buchet and Escolar provided 3 methods for constructing such $M'$ from $M$. We repeat these constructions and the proof of indecomposability, as our results rely heavily on such. Before we begin, we state two key facts which will allow us to easily define the maps $\phi^{M'}$. First, let $M,N$ be finitely generated $n$-dimensional persistence modules and suppose $M=\bigoplus\limits_{i=1}^m I_i^M$ and $N=\bigoplus\limits_{j=1}^p I_j^N$ are the indecomposeable summands of $M$ and $N$. Then  any homomorphism $f:M\rightarrow N$ can be summarized by a collection of maps $f_{ij}:I_i^M\rightarrow I_j^N$. Second, Lemma 2.3 implies the following for $n$-dimensional rectangular persistence modules:
\[   \Hom( K[\alpha_1,\beta_1], K[\alpha_2,\beta_2])\cong\left\{
\begin{array}{ll}
       K & \text{if }  \alpha_2\leq\alpha_1\leq\beta_2\leq\beta_1\\
      0 & \text{otherwise}. \\
 
\end{array}
\right. \]

 We call the first construction from [BE1] the \textbf{main rectangular construction}, which is as follows: Let the $j^{th}$ \textbf{layer} of an $n$-dimensional persistence module $M$ be the restriction of $M$ to the plane with final coordinate $j$. Layer $3$ of $M'$ is $M=\bigoplus\limits_{i=1}^m K[\alpha_i, \beta_i]=:\bigoplus\limits_{i=1}^m I_{3i}$. In layer $2$, place $\bigoplus\limits_{i=1}^m K[\alpha_i, \beta_i']=:\bigoplus\limits_{i=1}^m I_{2i}$ where $\{\beta_i'\}_i$ are chosen such that $\beta_i\leq\beta_i'<\beta_{i+1}'$ for all $i$ and $\frac{\alpha_i+\beta_i'}{2}\leq\beta_j'$ for all $i,j$ (where  our partial ordering on $\mathbb{N}^n$ is extended to all of $(\frac{1}{2}\mathbb{N})^n$ in the obvious way). The maps $\phi^{M'}$ between layer $2$ and layer $3$ are defined by the matrix \begin{equation*}\phi_2^{M'}=\begin{bmatrix}
f_{21}^{31} &  & &\\
 & f_{22}^{32} & &\\
 & & \ddots & \\
 & & & f_{2m}^{3m}
\end{bmatrix},\end{equation*} where $f_{ri}^{sj}$ is shorthand for $f_{I_{ri}}^{I_{sj}}$ (c.f. notation at the end of Section 2.2).

Let $\alpha^j$ denote the $j^{th}$ coordinate of $\alpha\in\mathbb{N}^n$. Define $\mu\in\mathbb{N}^n$ by $\mu^j=\max_i (\alpha_i+\beta_i')^j$. In layer $1$ of $M'$ place $\bigoplus\limits_{i=1}^m K[\alpha_i', \beta_i']=: \bigoplus\limits_{i=1}^m I_{1i}$ where $\alpha_i':= \mu-\beta_i'$. Notice that this causes $\alpha_i\leq \alpha_i'<\beta_j'$ for all $i,j$. Combining this with the fact that $\beta_i'<\beta_{i+1}'$ implies that  $\alpha_{i+1}'<\alpha_{i}'<\beta_i<\beta_{i+1}'$ for all $i$. We thus have strict inclusions $\Supp(I_{1i})\subsetneq \Supp(I_{1(i+1)})$ for all $i$, a detail that will be critical later. The maps $\phi^{M'}$ between layer $1$ and layer $2$ are defined by \begin{equation*}\phi_1^{M'}=\begin{bmatrix}
f_{11}^{21} &  & &\\
 & f_{12}^{22} & &\\
 & & \ddots & \\
 & & & f_{1m}^{2m}
\end{bmatrix}.\end{equation*}
Layer $0$ of $M'$ is $ K[\max_i(\alpha_i'), \max_i(\beta_i')]=I_{01}$, and the maps between layer $0$ and layer $1$ are given by \begin{equation*}\phi_0^{M'}=\begin{bmatrix}
f_{01}^{11}\\  f_{01}^{12} \\ \vdots \\ f_{01}^{1m}
\end{bmatrix}.\end{equation*}

The second construction from [BE1], which we call the \textbf{dual rectangular construction}, is as follows:  Layer $0$ of $M'$ is $M= \bigoplus\limits_{i=1}^m K[\alpha_i, \beta_i]=:\bigoplus\limits_{i=1}^m I_{0i}$. In layer $1$, place $\bigoplus\limits_{i=1}^m K[\alpha_i', \beta_i]=\bigoplus\limits_{i=1}^m I_{1i}$ where $\{\alpha_i'\}_i$ are chosen to be distinct elements of $\mathbb{N}^n$ such that $\alpha_i'\leq\alpha_i$ for all $i$ and $\alpha_i'\leq \frac{\alpha_j'+\beta_j}{2}$ for all $i,j$ (where our partial ordering on $\mathbb{N}^n$ has again been extended to a partial ordering on $(\frac{1}{2}\mathbb{N})^n$). Distinctness implies that, without loss of generality, $\alpha_i'<\alpha_{i+1}'$ for all $i$. 

\begin{remark} Notice that such restrictions on $\{\alpha_i'\}_i$ may force some $\alpha_i'$ to be negative, which is not allowed. In this case, simply shift $M'$ to the right until all $\{\alpha_i'\}_i$ are non-negative. Due to the triviality of such shifting issues, we will not mention them in the constructions given henceforth.\end{remark}

The maps $\phi^{M'}$ between layer $0$ and layer $1$ are given by \begin{equation*}\phi_0^{M'}=\begin{bmatrix}
f_{01}^{11} &  & &\\
 & f_{02}^{12} & &\\
 & & \ddots & \\
 & & & f_{0m}^{1m}
\end{bmatrix}.\end{equation*} Define $\mu\in\mathbb{N}^n$ by $\mu^j:=\min_i(\alpha_i'+\beta_i)^j$. Layer $2$ is given by $\bigoplus\limits_{i=1}^m K[\alpha_i', \beta_i']=:\bigoplus\limits_{i=1}^m I_{2i}$ where $\beta_i':=\mu-\alpha_i'$. This implies that $\alpha_i'<\alpha_{i+1}'<\beta_{i+1}'<\beta_{i}'$ for all $i$ since $\alpha_i'<\alpha_{i+1}'$. In particular, we have that $\Supp(I_{2(i+1)})\subsetneq \Supp(I_{2i})$ for all $i$. The maps $\phi^{M'}$ between layer $1$ and layer $2$ are summarized by \begin{equation*}\phi_1^{M'}=\begin{bmatrix}
f_{11}^{21} &  & &\\
 & f_{12}^{22} & &\\
 & & \ddots & \\
 & & & f_{1m}^{2m}
\end{bmatrix}.\end{equation*}
Layer $3$ of $M'$ is $ K[\min_i(\alpha_i'), \min_i(\beta_i')]:=I_{31}$. The maps from layer $2$ to layer $3$ are given by \begin{equation*}\phi_2^{M'}=\begin{bmatrix}
f_{21}^{31} & f_{22}^{31} & \hdots & f_{2m}^{31}
\end{bmatrix}.\end{equation*}

The third construction from [BE1], which we call the \textbf{glued rectangular construction}, is to place $M$ in layer $3$, then place the main rectangular construction for $M$ in layers $0-2$ and the dual rectangular construction for $M$ in layers $4-6$. For an example of these three constructions, see Figure \ref{fig:BuchetConstructions}.

\begin{figure}[h]
    \centering
    \includegraphics[scale=.27]{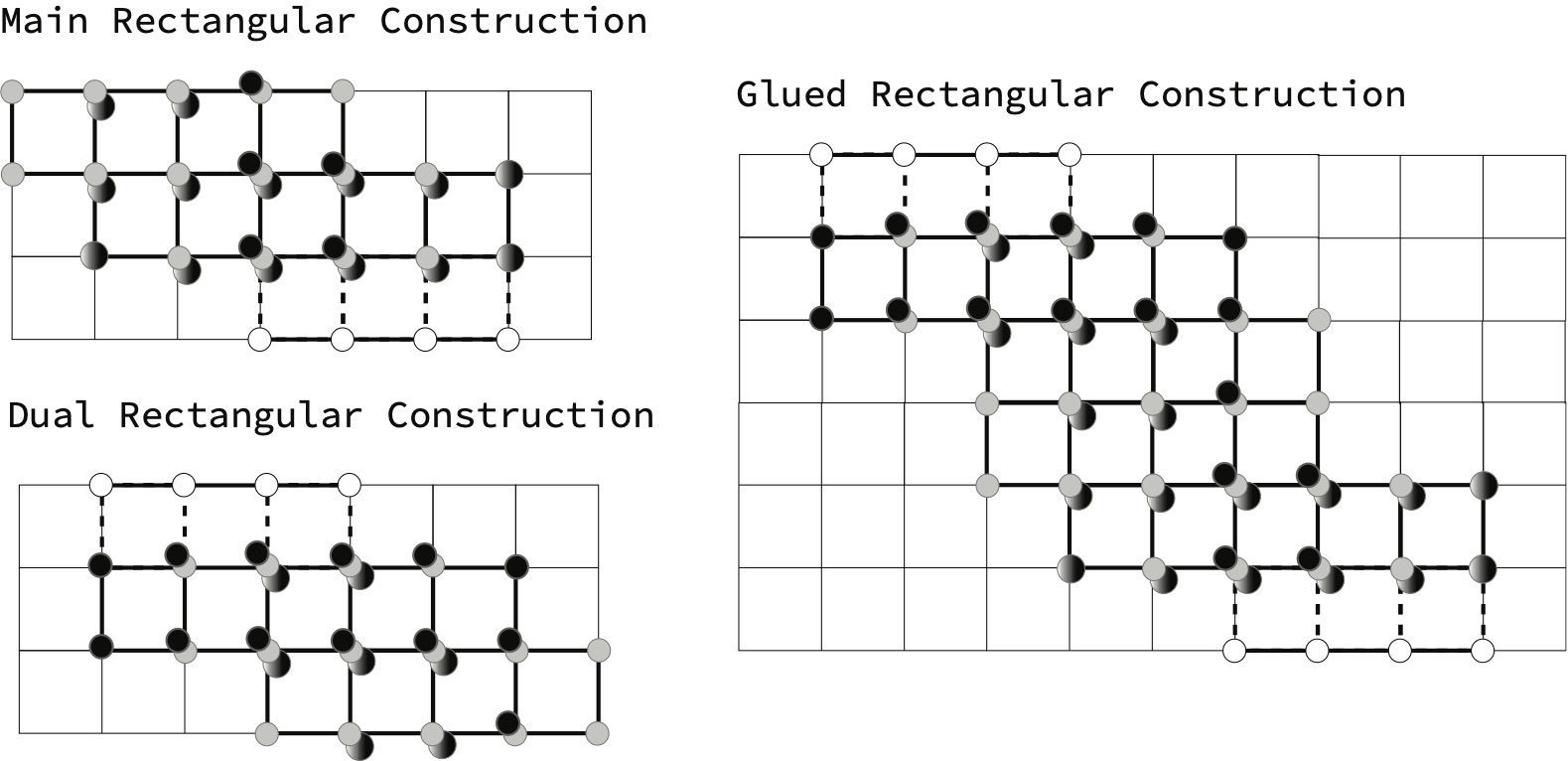}
    \vspace{.2cm}
    \caption{This is an example of the three constructions given in [BE1] for creating an indecomposable $2$-dimensional persistence module $M'$ from a finite $1$-dimensional persistence module $M$ such that $M$ is a hyperplane restriction of $M'$. In particular, these $2$-dimensional persistence modules come from the $1$-dimensional persistence module $M= K[3,3]\bigoplus K[0,4]\bigoplus K[1,2]$. In the main rectangular construction, dashed edges represent the map which sends the white basis vector to $(1,1,1)$. In the dual rectangular construction, dashed edges represent the map which sends each of the gray, black, and ombre basis vectors to the white basis vector.}
    \label{fig:BuchetConstructions}
\end{figure}

\begin{proof}
We prove that the dual rectangular construction results in a finite indecomposable $(n+1)$-dimensional persistence module $M'$ when $M=\bigoplus\limits_{i=1}^m K[\alpha_i, \beta_i]$ is a finite $n$-dimensional (rectangular) persistence module. The proofs involving the other constructions are similar.

Finiteness of $M'$ is clear. By Lemma 2.1, to show that $M'$ is indecomposable, it suffices to prove that $\End(M')\cong K$. Let $f\in\End(M')$. Then $f$ is equivalent to a set of endomorphisms $f_j$ on each layer $j$ of $M'$ such that the maps $\phi_i^{M'}$ and $f_j$ commute. In particular, we need $\phi_{i-1}^{M'} \circ f_{i-1}= f_i \circ \phi^{M'}_{i-1}$ for all $i$. By Lemma 2.3, $f_3=[cf_{31}^{31}]$ for some $c\in K$. We wish to show that \begin{equation*}f_i=\begin{bmatrix}
cf_{i1}^{i1} &  & &\\
 & cf_{i2}^{i2} & &\\
 & & \ddots & \\
 & & & cf_{im}^{im}
\end{bmatrix}\end{equation*} for all $i\leq 2$, which implies that $f$ is determined fully by the choice of $c\in K$.

Note that \begin{equation*}f_3\circ\phi_2^{M'}=\begin{bmatrix}
cf_{21}^{31} & cf_{22}^{31} &\dotsb& cf_{2m}^{31}
\end{bmatrix}.\end{equation*}  By Lemma 2.3, 
$\Hom(I_{2i}, I_{2j})\cong \begin{cases}
        K & \text {if } i=j  \\
      0 & \text{otherwise,}
    \end{cases}$
as $I_2i=K [\alpha_i', \beta_i']$ and $\alpha_i'<\alpha_{i+1}'<\beta_{i+1}'<\beta_{i}'$  for all $i$. As such, \begin{equation*}f_2=\begin{bmatrix}
c_{11}f_{21}^{21} &  & &\\
 & c_{22}f_{22}^{22} & &\\
 & & \ddots & \\
 & & & c_{mm}f_{2m}^{2m}
\end{bmatrix}\end{equation*} for some constants $c_{ii}\in K$, which implies that \begin{equation*}\phi_2^{M'}\circ f_2= \begin{bmatrix}
c_{11}f_{21}^{31} & c_{22}f_{22}^{31} &\dotsb& c_{mm}f_{2m}^{31}
\end{bmatrix}.\end{equation*} Since $f_3\circ\phi_2^{M'}=\phi_2^{M'}\circ f_2$, we may conclude that \begin{equation*}f_2=\begin{bmatrix}
cf_{21}^{21} &  & &\\
 & cf_{22}^{22} & &\\
 & & \ddots & \\
 & & & cf_{2m}^{2m}
\end{bmatrix},\end{equation*} as claimed.

Now we wish to show that \begin{equation*}f_1=\begin{bmatrix}
cf_{11}^{11} &  & &\\
 & cf_{12}^{12} & &\\
 & & \ddots & \\
 & & & cf_{1m}^{1m}
\end{bmatrix}.\end{equation*} By Lemma 2.3, $f_1=\begin{bmatrix} c_{jk}g_{1j}^{1k}
\end{bmatrix}_{j,k}$, where \[ g_{1j}^{1k}=\left\{
\begin{array}{ll}
      f_{1j}^{1k} & \text{if } \Hom(I_{1j}, I_{1k})\cong K \\
      0 & \text{if } \Hom(I_{1j}, I_{1k})=0 \\
 
\end{array}
\right. \]
 and each $c_{jk}\in K$. The same lemma also guarantees that $g_{ij}^{ij}=f_{ij}^{ij}$ for all $j$. That is, the diagonal entries of $f_1$ are nonzero (unless $c_{jj}=0)$, but the off-diagonal entries may or may not be zero. Then $$\phi_{1}^{M'} \circ f_{1}=\begin{bmatrix} c_{jk}g_{1j}^{2k}
\end{bmatrix}_{j,k}=\begin{bmatrix}
cf_{11}^{21} &  & &\\
 & cf_{12}^{22} & &\\
 & & \ddots & \\
 & & & cf_{1m}^{2m}
\end{bmatrix}= f_2\circ \phi^{M'}_{1}$$ implies that $c_{jk}g_{1j}^{2k}=c \delta_{j,k} f_{1j}^{2k}$ for all $j,k$. Thus $c_{jj}=c$ for all $j$. 

We claim that $c_{jk}g_{1j}^{1k}=0$ if $j\neq k$. If $g_{1j}^{1k}=0$, then we are done. Otherwise, we have that $g_{1j}^{1k}\neq 0$ but $c_{jk}g_{1j}^{2k}=c_{jk}g_{1k}^{2k}\circ g_{1j}^{1k}= 0$. Recall that $I_{1i}= K[\alpha_i',\beta_i]$ for each $i$. By Lemma 2.3, $g_{1j}^{1k}\neq 0$ implies that $\alpha_k'\leq \alpha_j'\leq\beta_k\leq\beta_j$. The definition of $\{\alpha_i'\}_i$ implies that $\alpha_k'< \alpha_j'\leq\beta_k\leq\beta_j$. Recall that $I_{2k}= K[\alpha_k',\beta_k']$ where $\beta_k\leq \beta_k'$. The definition of $\beta_k'$ further implies that $\alpha_k'< \alpha_j'<\beta_k'\leq\beta_j.$ By Lemma 2.3, there thus exists a nonzero homomorphism from $I_{1j}$ into $I_{2k}$. As such, $g_{1j}^{2k}=f_{1j}^{2k}\neq 0$. Thus $c_{jk}g_{1j}^{2k}=0$ implies $c_{jk}=0$, as desired.

It can similarly be shown that $$f_0=\begin{bmatrix}
cf_{01}^{01} &  & &\\
 & cf_{02}^{02} & &\\
 & & \ddots & \\
 & & & cf_{0m}^{0m}
\end{bmatrix}$$ by paying careful attention to the relationships between the intervals $I_{0i}, I_{0j}$ and $I_{1j}$  for $i\neq j$. Thus $f\in\End(M')$ is uniquely determined by the choice of $c\in K$.
\end{proof}

\section{Generalizations of Buchet and Escolar's result}
We now state our first result, which will be substantially strengthened later in Theorem 3.3.

\begin{theorem}
Let $M$ be a finite $n$-dimensional interval persistence module. Then there exists an indecomposable $(n+1)$-dimensional persistence module $M'$ such that $M$ is the restriction of $M'$ to a hyperplane. 
\end{theorem}

Based on Buchet and Escolar's constructions from [BE1], we have three methods for constructing such $M'$. The first method is the \textbf{main interval construction}, which is defined as follows: Place $M$ in layer $5$ of $M'$. Suppose $\bigoplus\limits_{i=1}^m I_{5i}$ is the decomposition of $M$ into indecomposable $n$D (interval) persistence modules. For $\alpha\in\mathbb{N}^n$, let $\alpha^j$ denote the $j^{th}$ coordinate of $\alpha$. For each $i$, define $\beta_i\in\mathbb{N}^n$ by $\beta_i^j=\max\limits_{\alpha\in\Supp(I_{5i})}(\alpha^j)$. It follows that $\alpha\leq \beta_i$ for all $\alpha\in I_{5i}$. Define $I_{4i}$ to be the indecomposable $n$-dimensional interval persistence module whose support is given by \begin{equation*}\Supp(I_{4i})=\{\gamma\in\mathbb{N}^n|\text{ there exists } \alpha\in\Supp(I_{5i}) \text{ satisfying } \alpha\leq\gamma\leq\beta_i \}.\end{equation*}

We claim that $\Hom(I_{4i}, I_{5i})\cong K$ for every $i$, which would allow us to define the maps $\phi^{M'}$ from layer $4$ to layer $5$ via $$\phi_4^{M'}=\begin{bmatrix}
f_{41}^{51} &  & &\\
 & f_{42}^{52} & &\\
 & & \ddots & \\
 & & & f_{4m}^{5m}
\end{bmatrix}.$$ By Lemma 2.3, it suffices to show that $\Supp(I_{5i})=\Supp(I_{4i})\cap\Supp(I_{5i})$ is a viable component. Suppose there exists $\gamma\in\Supp(I_{4i})\setminus \Supp(I_{5i})$ and $\delta\in \Supp(I_{5_i})$ such that $\gamma<\delta$. By the definition of $\Supp(I_{4i})$, $\gamma\in\Supp(I_{4i})$ implies that there exists $\alpha\in\Supp(I_{5i})$ such that $\alpha\leq\gamma<\delta$. By Remark 2.2, it follows that $\gamma\in\Supp(I_{5i})$, which is a contradiction. On the other hand, there are no points $\gamma\in\Supp(I_{5i})\setminus\Supp(I_{4i})$, so the second condition for non-viability also cannot exist in $\Supp(I_{4i})\cap\Supp(I_{5i})$.

Let $\{\alpha_{ij}\}_{j=1}^{\ell_i}$ denote the set of \textbf{minimal} elements in $I_{4i}$ (meaning there is no $\alpha\in \Supp(I_{4i})$ such that $\alpha<\alpha_{ij}$). Place $\bigoplus\limits_{i=1}^m\bigoplus\limits_{j=1}^{\ell_i}  K[\alpha_{ij}, \beta_i]=:\bigoplus\limits_{i=1}^m\bigoplus\limits_{j=1}^{\ell_i} I_{3ij} $ in layer $3$. Notice that $\Supp(I_{3ij})=\Supp(I_{3ij})\cap\Supp(I_{4i})$ cannot satisfy the first condition of non-viability, as there are no points $\gamma\in\Supp(I_{3ij})\setminus \Supp(I_{4i})$. On the other hand, suppose there exists $\gamma\in\Supp(I_{4i})\setminus\Supp(I_{3ij})$ and $\delta\in\Supp(I_{3ij})$ such that $\delta<\gamma$. Then $\delta<\gamma\leq \beta_i$. By Remark 2.2, $\delta, \beta_i\in\Supp(I_{3ij})$ implies that $\gamma\in\Supp(I_{3ij})$, which is a contradiction. Thus $\Supp(I_{3ij})=\Supp(I_{3ij})\cap\Supp(I_{4i})$ is a viable component, and  $\Hom(I_{3ij},I_{4i})\cong K$ for every $i\in[1,m]$ and $j\in[1,\ell_i]$. Define the maps $\phi^{M'}$ between layers $3$ and $4$ by \begin{equation*}\phi_3^{M'}:=\begin{bmatrix}
f_{311}^{41} & f_{312}^{41}  & \hdots &  f_{31\ell_1}^{41}&&&&&&&&&\\
&&&&f_{321}^{42}&f_{322}^{42}&\hdots &f_{32\ell_2}^{42}&&&&&\\
&&&&&&&&\ddots &&&&\\
&&&&&&&& &f_{3m1}^{4m}&f_{3m2}^{4m}&\hdots &f_{3m\ell_m}^{4m}
\end{bmatrix}.\end{equation*} For an example of constructing layers $3$ and $4$ of $M'$ from $M$, see Figure \ref{fig:MainIntervalConstruction}.

\begin{figure}[h]
    \centering
    \includegraphics[scale=.21]{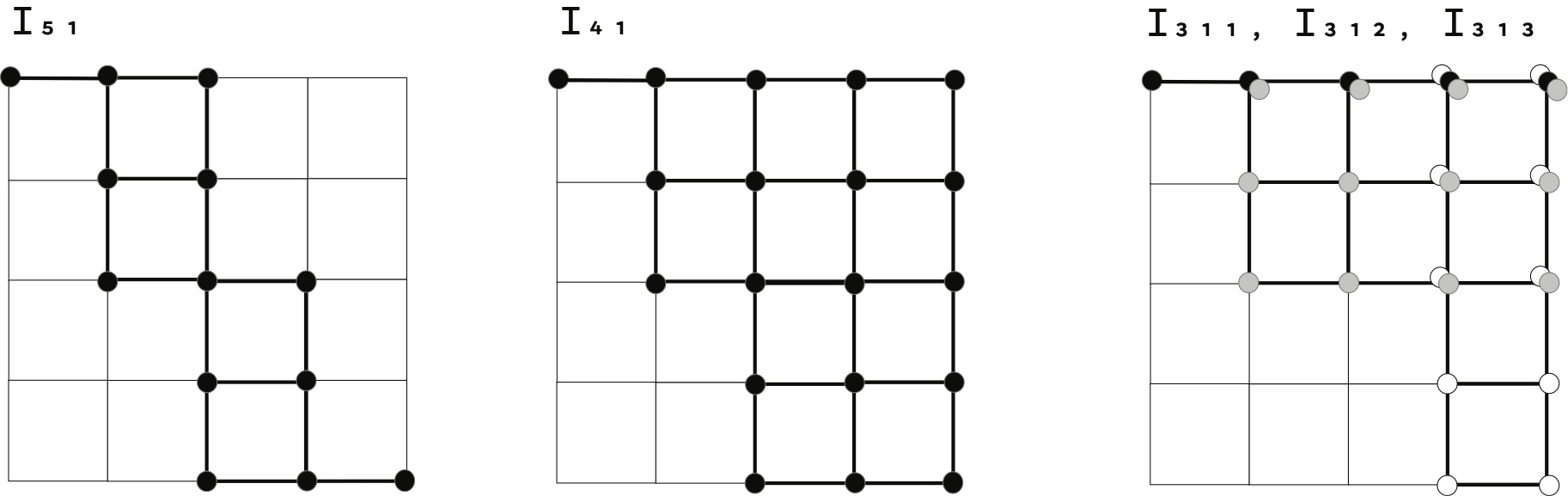}
    \vspace{.2cm}
    \caption{These are the top three layers of $M'$ in the main interval construction for $M=I_{51}$. In layer $5$, we have $M$. In layer $4$, the support of $I_{51}$ is extended to include $\beta_i=(4,4)$. In the third layer, we have three indecomposable rectangular $2$-dimensional persistence modules corresponding the three minimal elements of $\Supp(I_{51})$, namely $(2,0), (1,2),$ and $(0,4)$.}
    \label{fig:MainIntervalConstruction}
\end{figure}

Notice that layer $3$ is a finite rectangular $n$-dimensional persistence module (since $M$ was finite), so we may perform the main rectangular construction for layer $3$ in layers $0-2$. This finishes the main interval construction of $M'$.

The second construction is the \textbf{dual interval construction}. For this construction, place $M$ in layer $0$ of $M'$. Let $\bigoplus\limits_{i=1}^m I_{0i}$ be the decomposition of $M$ into its indecomposable summands. For each $i$, define $\alpha_i$ as the element in $\mathbb{N}^n$ such that $\alpha_i^j=\min\limits_{\beta\in\Supp(I_{0i})}(\beta^j)$. This yields that $\alpha_i\leq\beta$ for all $\beta\in\Supp(I_{0i})$. Define $I_{1i}$ as the indecomposeable interval $n$-dimensional persistence module with \begin{equation}\Supp(I_{1i}):=\{\gamma\in\mathbb{N}^n| \text{ there exists } \beta\in\Supp(I_{0i}) \text{ such that } \alpha_i\leq\gamma\leq\beta\}.\end{equation}  In layer $1$, place $\bigoplus\limits_{i=1}^m I_{1i}$.

We claim that $\Hom(I_{0i}, I_{1i})\cong K$ for all $i$. By Lemma 2.3, it suffices to show that $\Supp(I_{0i})= \Supp(I_{0i})\cap\Supp(I_{1i})$ is a viable component. There cannot be $\gamma\in\Supp(I_{0i})\setminus \Supp(I_{1i})$, so the first condition of non-viability cannot be satisfied. Suppose there exists $\gamma\in\Supp(I_{1i})\setminus\Supp(I_{0i})$ and $\delta\in I_{0i}$ such that $\delta<\gamma$. By equation (1), there exists $\beta\in\Supp(I_{0i})$ such that $\delta<\gamma\leq\beta$. Since $\delta,\beta\in\Supp(I_{0i})$, Remark 2.2 implies that $\gamma\in\Supp(I_{0i})$, which yields a contradiction. Thus $\Supp(I_{0i})= \Supp(I_{0i})\cap\Supp(I_{1i})$ is indeed a viable component, and the maps $\phi^{M'}$ between layers $0$ and $1$ of $M'$ are defined by $$\phi_0^{M'}=\begin{bmatrix}
f_{01}^{11} &  & &\\
 & f_{02}^{12} & &\\
 & & \ddots & \\
 & & & f_{0m}^{1m}
\end{bmatrix}.$$

For each $i$, let $\{\beta_{ij}\}_{j=1}^{\ell_i}$ denote the set of \textbf{maximal} elements in $\Supp(I_{1i})$ (meaning there is no $\beta\in\Supp(I_{1i})$ such that $\beta_{ij}<\beta$). In layer $2$, place $\bigoplus\limits_{i=1}^m \bigoplus\limits_{j=1}^{\ell_i}  K[\alpha_i, \beta_{ij}]=:\bigoplus\limits_{i=1}^m \bigoplus\limits_{j=1}^{\ell_i} I_{2ij}$. We claim that $\Hom(I_{1i}, I_{2ij})\cong K$ for all $i,j$. By Lemma 2.3, this follows if $\Supp(I_{2ij})=\Supp(I_{2ij})\cap\Supp(I_{1i})$ is a viable component. Suppose there exists $\gamma\in\Supp(I_{1i})\setminus \Supp(I_{2ij})$ and $\delta\in\Supp(I_{2ij})$ such that $\gamma<\delta$. Then  $\gamma\in\Supp(I_{1i})$ implies $\alpha_i\leq\gamma<\delta$. By Remark 2.2, because $\alpha_i, \delta\in\Supp(I_{2ij})$, we may conclude $\gamma\in\Supp(I_{2ij})$, which is a contradiction. Thus $\Supp(I_{2ij})$ fails the first condition for non-viability. The second condition for non-viability also cannot occur for $\Supp(I_{2ij})$, as $\Supp(I_{2ij})\setminus\Supp(I_{1i})=\emptyset$. Thus $\Supp(I_{2ij})$ is indeed a viable component and \begin{equation*}\phi_1^{M'}:=\begin{bmatrix}
f_{11}^{211} &  & &\\
 f_{11}^{212}& & &\\
\vdots &&&\\
 f_{11}^{21\ell_1}& & & \\
 &f_{12}^{221}  & &\\
 &f_{12}^{222} & &\\
&\vdots &&\\
 &f_{12}^{22\ell_2} & &  \\
 &&\ddots &\\
 & & & f_{1m}^{2m1}\\
  & & & f_{1m}^{2m2}\\
  &&&\vdots\\
   & & & f_{1m}^{2m\ell_m}
\end{bmatrix}\end{equation*} is a valid way of defining the maps $\phi^{M'}$ between layers $1$ and $2$ of $M'$. See Fig. \ref{fig:DualIntervalConstruction}. for an example of constructing $I_{1i}$ and $\{I_{2ij}\}_j$ from $I_{0i}$.

\begin{figure}[h]
    \centering
    \includegraphics[scale=.21]{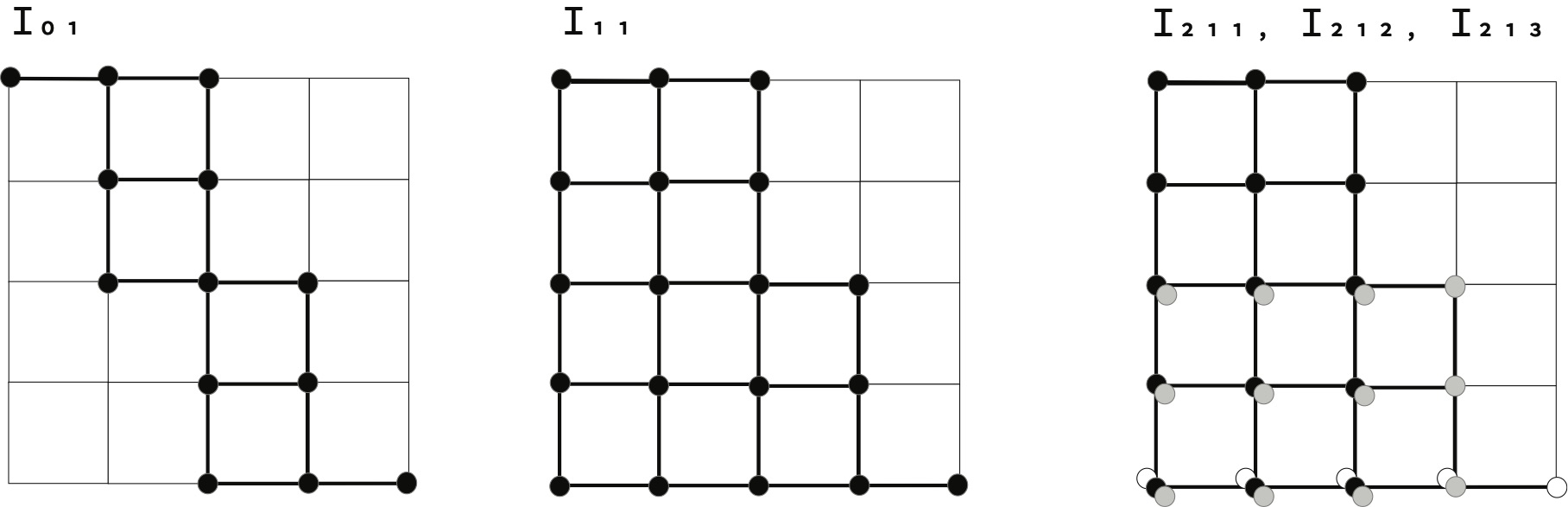}
    \vspace{.2cm}
    \caption{This is an example of the first three layers of $M'$ in the dual interval construction. In the zeroeth layer of $M'$, we have $M=I_{01}$. In the first layer of $M'$, we extend the support of $I_{01}$ to include $\alpha_1=(0,0)$. In the third layer, we have $I_{211}\bigoplus I_{212} \bigoplus I_{213}= K[(0,0),(2,4)]\bigoplus K[(0,0),(3,2)]\bigoplus K[(0,0), (4,0)]$, as $(2,4), (3,2),$ and $(4,0)$ are the maximal elements in $\Supp(I_{11})$.}
    \label{fig:DualIntervalConstruction}
\end{figure}

Notice that because $M$ is finite, layer 2 of $M'$ is a finite rectangular $n$-dimensional persistence module. Thus, we may perform the dual rectangular construction for the second layer of $M'$ in the planes above the second layer. This finishes the dual interval construction for $M'$. The \textbf{glued interval construction} of $M'$ is given by placing $M$ in layer 5 of $M'$ and then doing the main interval construction for $M$ in the layers below $M$, and the dual interval construction for $M$ in the layers above $M$.

\begin{proof}
Given $M$ as in Theorem 3.1, construct $M'$ using the dual interval construction. We now prove Theorem 3.1 by showing that $M'$ is indecomposable. The proofs for the main and glued interval constructions are similar.

Let $f\in\End(M')$. Then $f$ is equivalent to a set of endomorphisms $f_j$ on each layer $j$ of $M'$ such that $\phi_{i-1}^{M'} \circ f_{i-1}= f_i \circ \phi^{M'}_{i-1}$ for all $i$. By the proof of Theorem 2.5,  we know that $f_5=[cf_{51}^{51}]$ for some $c\in K$ and $$f_i=\begin{bmatrix}
cf_{i1}^{i1} &  & &\\
 & cf_{i2}^{i2} & &\\
 & & \ddots & \\
 & & & cf_{im}^{im}
\end{bmatrix}$$ for all $i\in[2,4]$. We wish to show that $f_i$ is also determined by $c$ when $i=0,1$.

Notice that $\Supp(I_{1j})\cap\Supp( I_{1k})$ has either zero components or one component due to the following: Suppose, without loss of generality, that $\alpha_{j}\leq\alpha_{k}$ and $\Supp(I_{1j})\cap\Supp( I_{1k})\neq\emptyset$. This implies that  $\alpha_{k}\leq\gamma$ for all $\gamma\in\Supp(I_{1j})\cap\Supp( I_{1k})$, since  $\alpha_{k}\leq\gamma$ for all $\gamma\in\Supp(I_{1k})$.  Then $\gamma,\delta\in\Supp(I_{1j})\cap\Supp( I_{1k})$ implies that $\alpha_k\leq \gamma,\delta$. By Remark 2.2, any element $\epsilon\in\mathbb{N}^n$ satisfying $\alpha_k\leq\epsilon\leq\gamma$ must be in both $\Supp(I_{1j})$ and $\Supp( I_{1k})$, so there is a path from $\alpha_k$ to $\gamma$ in $\Supp(I_{1j})\cap\Supp( I_{1k})$. Similarly, there is a path from $\alpha_k$ to $\delta$ in $\Supp(I_{1j})\cap\Supp( I_{1k})$. As such, $\gamma$ and $\delta$ are in the same component of $\Supp(I_{1j})\cap\Supp( I_{1k})$, implying that $\Supp(I_{1j})\cap\Supp( I_{1k})$ has at most one component. By Lemma 2.3, it follows that $f_1=[c_{ij}g_{1i}^{1j}]_{i,j}$ where \[ g_{1j}^{1k}=\left\{
\begin{array}{ll}
      f_{1j}^{1k} & \text{if } \Hom(I_{1j}, I_{1k})\cong K \\
      0 & \text{if } \Hom(I_{1j}, I_{1k})=0. \\
 
\end{array}
\right. \] Notice that $g_{1i}^{1i}=f_{1i}^{1i}$ is guaranteed for all $i$ by Lemma 2.3. Then  $$f_2 \circ \phi^{M'}_{1}=\begin{bmatrix}
cf_{11}^{211} &  & &\\
 cf_{11}^{212}& & &\\
\vdots &&&\\
 cf_{11}^{21\ell_1}& & & \\
 &cf_{12}^{221}  & &\\
 &cf_{12}^{222} & &\\
&\vdots &&\\
 &cf_{12}^{22\ell_2} & &  \\
 &&\ddots &\\
 & & & cf_{1m}^{2m1}\\
  & & & cf_{1m}^{2m2}\\
  &&&\vdots\\
   & & & cf_{1m}^{2m\ell_m}
\end{bmatrix}= \begin{bmatrix}
c_{11}f_{11}^{211} & c_{21} fg_{12}^{211}  & &\hdots &c_{m1} fg_{1m}^{211}\\
 c_{11}f_{11}^{212}& c_{21} fg_{12}^{212}& &&\vdots\\
\vdots &\vdots&&&\vdots\\
 c_{11}f_{11}^{21\ell_1}& c_{21} fg_{12}^{21\ell_1}&& &\vdots \\
 c_{12}fg_{11}^{221}&c_{22}f_{12}^{221}  & &&\vdots\\
 c_{12}fg_{11}^{222} &c_{22}f_{12}^{222} & &&\vdots\\
\vdots&\vdots &&&\vdots\\
 c_{12}fg_{11}^{22\ell_2} &c_{22}f_{12}^{22\ell_2} & &&\vdots  \\
 \vdots& \vdots&\ddots &&\vdots\\
 \vdots&\vdots& && c_{mm}f_{1m}^{2m1}\\
 \vdots &\vdots& && c_{mm}f_{1m}^{2m2}\\
 \vdots &\vdots&&&\vdots\\
  c_{1m} fg_{11}^{2m\ell_m} &\hdots&\hdots&\hdots& c_{mm}f_{1m}^{2m\ell_m}
\end{bmatrix}=\phi_{1}^{M'} \circ f_{1},$$
where $fg_{1i}^{2jk}:=f_{1j}^{2jk}\circ g_{1i}^{1j}$. Thus $c_{ii}=c$ for all $i$ and $c_{ij}f_{1j}^{2jk}\circ g_{1i}^{1j}=0$ whenever $i\neq j$ and $k\in[1, j_\ell]$. 

We claim that $c_{ij}g_{1i}^{1j}=0$ for all $i\neq j$. If $g_{1i}^{1j}=0$, we are done. Otherwise, $g_{1i}^{1j}$ is a nonzero homomorphism from $I_{1i}$ to $I_{1j}$, implying (by Lemma 2.3) that $\alpha_i\leq\alpha_j$ and $\Supp(I_{1i})\cap\Supp(I_{1j})\neq\emptyset$. As such, there exists a maximal element $\beta_{jk}\in\Supp(I_{1j})$ such that $\Supp(I_{1i})\cap K[\alpha_j,\beta_{jk}]=\Supp(I_{1i})\cap\Supp(I_{2jk})\neq\emptyset$. It follows that $f_{1j}^{2jk}\circ g_{1i}^{1j}=f_{1i}^{2jk}\neq 0$ on $\Supp(I_{1i})\cap\Supp(I_{2jk})$. Combining this with the fact that $c_{ij}f_{1j}^{2jk}\circ g_{1i}^{1j}=0$  implies that $c_{ij}=0$. Thus $$f_1=\begin{bmatrix}
cf_{11}^{11} &  & &\\
 & cf_{12}^{12} & &\\
 & & \ddots & \\
 & & & cf_{1m}^{1m}
\end{bmatrix}$$ is fully determined by $c$, as claimed.

Next, we consider the relationship $\phi_{0}^{M'} \circ f_{0}= f_1 \circ \phi^{M'}_{0}$. By above, we have $$f_1 \circ \phi^{M'}_{0}=\begin{bmatrix}
cf_{01}^{11} &  & &\\
 & cf_{02}^{12} & &\\
 & & \ddots & \\
 & & & cf_{0m}^{1m}
\end{bmatrix}.$$ Notice that $f_0=[c_{ij}G_{0i}^{0j}]_{i,j}$ where each $G_{0i}^{0j}\in\Hom(I_{0i}, I_{0j})$ and $\Hom(I_{0i}, I_{0j})\cong  K^k$ or $0$ (by Lemma 2.3). Unlike the settings we've previously worked with, $k>1$ is possible, depending on the shapes of $I_{0i}$ and $I_{0j}$. Lemma 2.3 further implies that $G_{0i}^{0i}=f_{0i}^{0i}$ for each $i$. Then $\phi_{0}^{M'} \circ f_{0}= [c_{ij}f_{0j}^{1j}G_{0i}^{0j}]_{i,j}$. Comparing with $f_1 \circ \phi^{M'}_{0}$ immediately gives $c_{ii}=c$ for all $i$ and $c_{ij}f_{0j}^{1j}G_{0i}^{0j}=0$ for all $i\neq j$. 

We claim that $c_{ij}G_{0i}^{0j}=0$ for all $i\neq j$. If $G_{0i}^{0j}=0$, we are done. Otherwise, by Lemma 2.3, there exists a viable component $C$ of $\Supp(I_{0i})\cap\Supp(I_{0j})$ such that $G_{0i}^{0j}|_C=d \Id_{ K}$ for some nonzero constant $d\in K$. Thus $c_{ij}f_{0j}^{1j}G_{0i}^{0j}|_C=c_{ij}d \Id_{ K}$. Recall that we also know $c_{ij}f_{0j}^{1j}G_{0i}^{0j}=0$, so it must be the case that $c_{ij}=0$. Thus, $c_{ij}G_{0i}^{0j}=0$ for all $i\neq j$ and $$f_0=\begin{bmatrix}
cf_{01}^{01} &  & &\\
 & cf_{02}^{02} & &\\
 & & \ddots & \\
 & & & cf_{0m}^{0m}
\end{bmatrix}$$ is fully determined by the choice of $c$.

%Because of how $I_{1j}$ is created from $I_{0j}$, it will be the case that $C$ is a viable component of $\Supp(I_{0i})\cap\Supp(I_{1j})$ as well. Thus $\Hom(I_{0i}, I_{1j})\neq 0$, so it is possible for $f_{0j}^{1j}\circ G_{0i}^{0j}$ to be nonzero. 

\end{proof}

Theorem 3.1 has an interesting consequence on zigzag persistence modules, which are a generalization of $1$-dimensional persistence modules that were first introduced in [CdS].

A \textbf{zigzag persistence module $M$} over $ K$ is an assignment of a vector space $M_\alpha$ to each $\alpha\in\mathbb{N}$ and for each $\alpha\in\mathbb{N}$, a homomorphism of the form $\phi^M_{\alpha, \alpha+1}:M_\alpha\rightarrow M_{\alpha+1}$ \underline{or} of the form $\phi^M_{\alpha+1, \alpha}:M_{\alpha+1}\rightarrow M_{\alpha}$. A map of the form $\phi^M_{\alpha, \alpha+1}:M_\alpha\rightarrow M_{\alpha+1}$ is said to be \textbf{forwards-oriented}, whereas a map of the form $\phi^M_{\alpha+1, \alpha}:M_{\alpha+1}\rightarrow M_{\alpha}$ is \textbf{backwards-oriented}.

Thus a $1$-dimensional persistence module is a zigzag module in which all maps are forward-oriented. The definitions of \textbf{finitely generated} and \textbf{finite} zigzag persistence modules are analagous to the definitions in the $n$D persistence module case. Finite zigzag modules also have indecomposeable decompositions (which are again defined analagously to the nD persistence module case). A zigzag module $M$ is indecomposeable if and only if $\Supp(M)=[\alpha,\beta]$ for some $\alpha, \beta$ with $\dim(M_\gamma)= \begin{cases}
        1 & \text {if } \alpha\leq\gamma\leq\beta \\
      0 & \text{otherwise,}
    \end{cases}$ and $\phi^M_{\gamma,\delta}=\begin{cases}
        \Id_K & \text {if } \alpha\leq\gamma\leq\delta\leq\beta  \\
      0 & \text{otherwise}
    \end{cases}$ [G].

\begin{corollary}
Let $M=\bigoplus\limits_{i=1}^m K[\alpha_i, \beta_i]$ be a finite zigzag persistence module. Then there exists an indecomposable $3$-dimensional persistence module $M'$ such that $M$ is the restriction of $M'$ to a path. 
\end{corollary}

\begin{proof} Let $Q_2$ denote the graph whose vertices are the elements of $\mathbb{N}^2$ and whose edges are of the form $\alpha\rightarrow\alpha_i$ for all $\alpha\in\mathbb{N}^n$ and $i\in[1,2]$. Place $M$ in $Q_2$ such that the backwards oriented arrows of $M$ go along arrows of the form $\alpha\rightarrow\alpha+e_2$ in$Q_2$ and forwards oriented arrows in $M$ go along arrows of the form $\alpha\rightarrow\alpha+e_1$ in $Q_2$. See Fig. \ref{fig:PlacingZigzags} for examples of these placements. The result is a finite $2$-dimensional interval persistence module. Apply Theorem 2.1.

\begin{figure}[h]
    \centering
    \includegraphics[scale=.19]{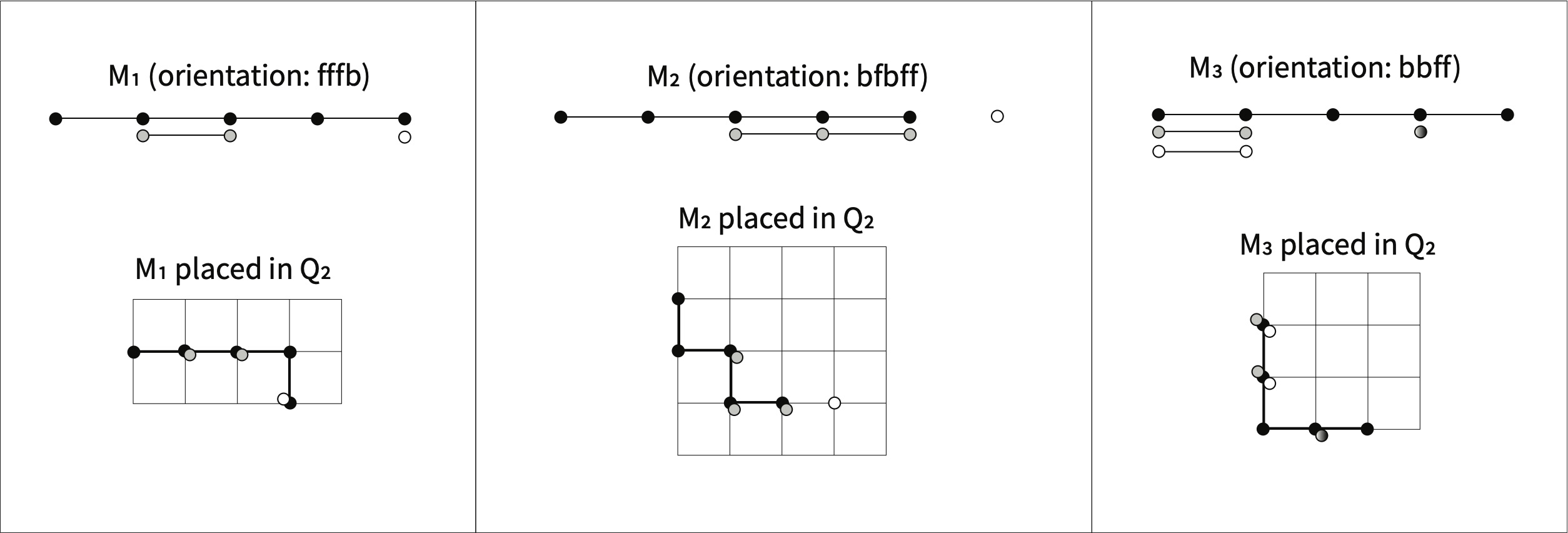}
    \vspace{.2cm}
    \caption{Shown are three examples of how to place a zigzag module $M_i$ onto $Q_2$. The orientation of each $M_i$ is given, with `f' indicating a forward oriented arrow, and `b' indicating backwards oriented arrows. For example, the first four maps in $M_1$ are forwards-oriented, but the map between the final two coordinates is backwards-oriented. Similar to our depictions of $n$-dimensional persistence modules, each dot at $\alpha$ represents a basis vector in $M_\alpha$. The homomorphisms in $M_i$ are denoted by edges; an edge between dots indicates which basic vectors map to each other.}
    \label{fig:PlacingZigzags}
\end{figure}
\end{proof}

\begin{theorem}
Let $M$ be any finite $n$-dimensional persistence module. Then there exists an indecomposable $(n+1)$-dimensional persistence module $M'$ such that $M$ is the restriction of $M'$ to a hyperplane. 
\end{theorem}

%\begin{remark} The conditions of theorem 3.3 do not hold for every finite $n$-dimensional persistence module $M$. For example, the $2$-dimensional persistence module shown in Fig. \ref{fig:IndecomposableExample} is indecomposeable for each choice of $d>1\in\mathbb{N}$ and $\lambda\in K=\mathbb{R}$, but $\End(M)\ncong  K$. In fact, [realizations of...] showed (using methods similar to those in this paper) that $$\End(M)= \left\{ d\times d \text{ matrices } C \text{ over }  K\Bigg|C=\begin{bmatrix}
%c_1&c_2  &\hdots &c_d\\
% & c_1 & \ddots&\vdots\\
% & & \ddots &c_2 \\
% & & & c_1
%\end{bmatrix} \text{ for some choices } c_i\in K\right\}.$$ An element $C$ in $\End(M)$ is invertible if and only if $c_1\neq 0$. Thus $C$ or $I-C$ is be invertible, and $\End(M)$ is indeed local but not isomorphic to $K$. On the other hand, if $d=1$, the same proof yields that $\End(M)\cong K$ for any choice of $\lambda\in\mathbb{R}$.
%\end{remark}
%

We have one construction for $M'$, called the \textbf{main construction}, which is defined as follows: Place $M=\bigoplus\limits_{i=1}^m I_{6i}$ in layer $6$ of $M'$. Choose a minimal generating set $\{g_{6ij}\}_{j=1}^{k_i}$ for $I_{6i}$, and let $g_{6ij}\in (I_{6i})_{\alpha_{ij}}$ for all $i,j$. Let $I_{5ij}$ denote the span of $g_{6ij}$ in $I_{6i}$, meaning 

\[ (I_{5ij})_\alpha:=\left\{
\begin{array}{ll}
      \Span(\phi^{I_{6i}}_{\alpha_{ij},\alpha}(g_{6ij}))\subset (I_{6i})_\alpha & \text{if } \alpha_{ij}\leq\alpha \\
      0 & \text{otherwise} \\
\end{array}
\right. \]
and
\[ \phi^{I_{5ij}}_{\alpha,\beta}=\left\{
\begin{array}{ll}
      \Id_ K & \text{if } (I_{5ij})_\alpha\cong (I_{5ij})_\beta\cong K \\
      0 & \text{otherwise.} \\
\end{array}
\right. \] Notice that  $I_{5ij}$ is a finite indecomposable interval $n$-dimensional persistence module with a single generator $g_{5ij}$. For each $i,j$ define a homomorphism 
 $A_{5ij}^{6i}:I_{5ij}\rightarrow I_{6i}$ by $A_{5ij}^{6i}(g_{5ij})=g_{6ij}$. In layer $5$ of $M'$, place $\bigoplus\limits_{i=1}^m\bigoplus\limits_{j=1}^{k_i} I_{5ij}$ and define the maps $\phi^{M'}$ between layers $5$ and $6$ by
$$\phi^{M'}_5=\begin{bmatrix}
A_{511}^{61} & A_{512}^{61} & \hdots & A_{51k_1}^{61} & &&&&&&&&\\
&&&&A_{521}^{62} &A_{522}^{62}&\hdots&A_{52k_2}^{62}&&&&&\\
&&& & &&&&\ddots &&&&\\
&&& & &&&&&A_{5m1}^{6m}&A_{5m2}^{6m}&\hdots&A_{5mk_m}^{6m}\\
\end{bmatrix}.$$

Because layer $5$ is a finite interval $n$-dimensional persistence module, we may attach the main interval construction for layer $5$ of $M'$ in layers $0-4$. This completes our construction of $M'$.

\begin{proof}
We now show that this construction yields a finite indecomposable $(n+1)$-dimensional persistence module $M'$. Finiteness is trivial. Let $f\in\End(M')$. Then $f$ is equivalent to a set of endomorphisms $f_j$ on each layer $j$ of $M'$ such that $\phi_{i-1}^{M'} \circ f_{i-1}= f_i \circ \phi^{M'}_{i-1}$ for all $i$. By Lemma 2.3, $f_0=[cf_{01}^{01}]$ for some $c\in K$. By a proof similar to that of Theorem 3.1, this implies that $f_i$ is fully determined by $c$ for all $i\in[1,5]$. In particular, 
$$f_5=\begin{bmatrix}
cf_{511}^{511} &  & &&&&&&&&&\\
 & cf_{512}^{512} & &&&&&&&&&\\
  & &\ddots &&&&&&&&&\\
   & & &cf_{51k_1}^{51k_1}&&&&&&&&\\
      & & &&cf_{521}^{521}&&&&&&&\\
            & & &&&cf_{522}^{522}&&&&&&\\
                & & &&&&\ddots&&&&&\\
             & & &&&&&cf_{52k_2}^{52k_2}&&&&\\
                   & & &&&&&&cf_{5m1}^{5m1}&&&\\
                               & & &&&&&&&cf_{5m2}^{5m2}&&\\
                                & & &&&&&&&&\ddots&\\
                                    & & &&&&&&&&&cf_{5mk_m}^{5mk_m}\\
   
\end{bmatrix}.$$ 

We may write $f_6=[c_{ij}G_{6i}^{6j}]_{i,j}$ where $G_{6i}^{6j}\in \Hom(I_{6i}^{6j})$ for all $i,j$. 

We thus have $\phi^{M'}_5\circ f_5= c\phi^{M'}_5= f_6\circ\phi^{M'}_5$\\
$=\begin{bmatrix}
H_{511}^{61}&H_{512}^{61}&\hdots&H_{51k_1}^{61}&H_{521}^{61}&H_{522}^{61}&\hdots&H_{52k_2}^{61}&\hdots&H_{5m1}^{61}&H_{5m2}^{61}&\hdots&H_{5mk_m}^{61}\\
H_{511}^{62}&H_{512}^{62}&\hdots&H_{51k_1}^{62}&H_{521}^{62} &H_{522}^{62}&\hdots&H_{52k_2}^{62}&\hdots&H_{5m1}^{62}&H_{5m2}^{62}&\hdots&H_{5mk_m}^{62}\\
\vdots&\vdots&&\vdots & \vdots&\vdots&&\vdots& &\vdots&\vdots&&\vdots\\
H_{511}^{6m}&H_{512}^{6m}&\hdots&H_{51k_1}^{6m}&H_{521}^{6m}&H_{522}^{6m}&\hdots&H_{52k_2}^{6m}&\hdots&H_{5m1}^{6m}&H_{5m2}^{6m}&\hdots&H_{5mk_m}^{6m}\\
\end{bmatrix},$ \\
where $H_{5i\ell}^{6j}:=c_{ij}G_{6i}^{6j}A_{5i\ell}^{6i}$ for all $i,j$. We conclude that \begin{equation}c_{ii}G_{6i}^{6i}A_{5i\ell}^{6i}=cA_{5i\ell}^{6i}\end{equation} for all $i$ and \begin{equation}c_{ij}G_{6i}^{6j}A_{5i\ell}^{6i}=0\end{equation} whenever $i\neq j$ and $\ell\in[1,k_i]$. 

We first wish to show that $c_{ii}G_{6i}^{6i}=c\Id_{I_{6i}}$ for all $i$. By equation (2), we have that \linebreak $c_{ii}G_{6i}^{6i}A_{5i\ell}^{6i}(g_{5i\ell})=cA_{5i\ell}^{6i}(g_{5i\ell})$ for all $i, \ell$. By the definition of $A_{5i\ell}^{6i}$, this implies $c_{ii}G_{6i}^{6i}(g_{6i\ell})=cg_{6i\ell}$ for all $i,\ell$. Because $\{g_{6i\ell}\}_{\ell}$ generates $I_{6i}$, we may conclude that $c_{ii}G_{6i}^{6i}=c\Id_{I_{6i}}$, as desired.

We also claim that that $c_{ij}G_{6i}^{6j}=0$ for all $i\neq j$. If $G_{6i}^{6j}=0$, we are done. Otherwise, there is a generator $g_{6i\ell}$ of $I_{6i}$ such that  $G_{6i}^{6j}(g_{6i\ell})\neq 0$. Thus $G_{6i}^{6j}A_{5i\ell}^{6i}(g_{5i\ell})\neq 0$. Comparing to equation (3) yields that $c_{ij}=0$. Thus $c_{ij}G_{6i}^{6j}=0$ for all $i\neq j$.

Thus $f_6=c\Id_M$ and $f$ is fully determined by the choice of $c\in K$, implying $\End(M')\cong K$. By Lemma 2.1, $M'$ is indecomposeable.
\end{proof}

\section{Future Research Directions}
Of course, it may be possible to strengthen the results in this paper. Some open questions are: Can Corollary 3.2 be strengthened to a statement that any finite zigzag persistence module can be embedded into an indecomposable $2$-dimensional persistence module, rather than a $3$D persistence module? Given a finite $n$-dimensional persistence module $M$, is there an indecomposable $(n+1)$-dimensional persistence module $M''$ which has $M$ as a hyperplane restriction and has a 'smaller' support than $M'$ in some sense? Is Theorem 3.3 generalizeable to $n$D persistence modules over $\mathbb{R}^n$ rather than $\mathbb{N}^n$?  Note that $n$D persistence modules over $\mathbb{R}^n$ have very different properties than the $n$-dimensional persistence modules over $\mathbb{N}^n$, so such a generalization may be highly nontrivial [M].


\begin{thebibliography}{[AMSS2]}

\raggedbottom
\bibitem[CZ]{CZ} G. Carlsson, A. Zomorodian. The Theory of Multidimensional Persistence. Disc. Comput. Geom. 71-93, 2009.
\bibitem[BE1]{BE1} M. Buchet, E.G. Escolar. Every 1D Persistence Module is a Restriction of Some Indecomposable 2D Persistence Module. J. Appl. and Comput. Top. 2020.
\bibitem[LW]{LW} M. Lesnick, M. Wright. Interactive Visualization of 2-D Persistence Modules. arXiv: 1512.00180, preprint. 2015.
\bibitem[ABENY]{ABENY} H. Asashiba, M. Buchet, E.G. Escolar, K. Nakashima,
M. Yoshiwaki. On Interval Decomposability of 2d Persistence Modules. arXiv:1812.05261, preprint. 2018.
\bibitem[AENY]{AENY} H. Asashiba, E.G. Escolar, K. Nakashima,
M. Yoshiwaki. On Approximation of 2D Persistence Modules by Interval-Decomposables. 	arXiv:1911.01637, preprint. 2019.
\bibitem[B]{B} H.B. Bjerkevik. Stability of Higher-Dimensional Interval Decomposable Persistence Modules. arXiv:1609.02086, preprint. 2016.
\bibitem[DX]{DX} T.K Dey, C. Xin. Computing Bottleneck Distance for Multi-parameter
Interval Decomposable Modules. arXiv:1803.02869, preprint. 2018.
\bibitem[A]{A} I. Assem, A. Skowronski, D. Simson Elements of the Representation Theory of Associative Algebras: Volume 1: Techniques of Representation Theory. Cambridge University Press. 2006.
\bibitem[G]{G} P. Gabriel. Unzerlegbare darstellungen I. Manuscripta Mathematica. 71-103, 1972.
\bibitem[BE2]{BE2}M. Buchet, E.G. Escolar. Realizations of Indecomposable Persistence Modules of Arbitrarily Large Dimension. Symp. Comput. Geom. 2018.
\bibitem[CdS]{CdS} G. Carlsson, V. de Silva. Zigzag Persistence. Found. Comput. Math. 367-405, 2010.
\bibitem[M]{M} E. Miller. Homological Algebra of Modules over Posets. 	arXiv:2008.00063, preprint. 2020.
\end{thebibliography}
\end{document}